\documentclass{kms-b}

\issueinfo{}% volume number
  {}%        % issue number
  {}%        % month
  {}%     % year
\pagespan{1}{}
\copyrightinfo{}%              % copyright year
  {Korean Mathematical Society}% copyright holder
\usepackage{graphicx}
\usepackage{amssymb}
\usepackage{mathrsfs}
\usepackage[all]{xy}
\usepackage{tikz}
\allowdisplaybreaks

\numberwithin{equation}{section}

\theoremstyle{plain}
\newtheorem{theorem}{Theorem}[section]

\newtheorem{lemma}[theorem]{Lemma}

\theoremstyle{definition}
\newtheorem{definition}{Definition}

\theoremstyle{remark}

\begin{document}

\title[Hecke Curves in Frobenius Strata]
{Hecke curves in Frobenius strata of moduli space of
rank 2 vector bundles}

\author[L. Li]{Lingguang Li}
\address{Lingguang Li \\School of Mathematical Sciences, Tongji University\\  Key Laboratory of Intelligent Computing and Applications (Tongji University), \\ Ministry of Education\\ Shanghai 200092, China}
\email{LiLg@tongji.edu.cn}

\author[H. Zhang]{Hongyi Zhang}
\address{Hongyi Zhang \\ School of Mathematical Sciences, Tongji University \\ Key Laboratory of Intelligent Computing and Applications (Tongji University), \\ Ministry of Education\\ Shanghai 200092, China}
\email{2211076@tongji.edu.cn}
\thanks{This work was supported by National Natural Science Foundation of China (Grant No. 12171352), Applied Basic Research Programs of Science and Technology Commission Foundation of Shanghai Municipality (22JC1402700)}


\begin{abstract}
Let $k$ be an algebraically closed field with characteristic $2$, and let $X$ be a smooth projective algebraic curve of genus $g \geqslant 2$ over $k$. Let $\mathcal{M}^s_X(2,\mathcal{L})$ be the moduli space of rank $2$ stable vector bundles with determinant $\mathcal{L}$ on $X$. The Frobenius stratification measures the instability of
bundles in $\mathcal{M}^s_X(r,\mathcal{L})$ under pullback by the Frobenius map. We show that there exists a Frobenius stratum in $\mathcal{M}^s_X(2,\mathcal{L})$ which is covered by Hecke curves.
\end{abstract}

\maketitle

\section{Introduction}\label{intro}
Let $k$ be an algebraically closed field, and let $X$ be a smooth projective curve of genus $g \geqslant 2$ over $k$. Let $r\geqslant 2$ and $d$ be integers. Following the convention in the classical reference \cite{LeP}, we denote by $\mathcal{M}^s_X(r,d)$ the moduli space of stable vector bundles of rank $r$ and degree $d$ on $X$. A point $[E] \in \mathcal{M}^s_X(r,d)$ represents the isomorphism class of a stable vector bundle $E$ of rank $r$ and degree $d$. For any line bundle $\mathcal{L} \in \mathrm{Pic}^d(X)$, $\mathcal{M}_X^s(r,\mathcal{L})$ is the moduli space of stable vector bundles of rank $r$ with determinant $\mathcal{L}$ over $X$.  

If $\mathrm{char}(k)=p>0$, there exists a natural absolute Frobenius morphism $F:X \to X$ induced by $F^{\#}:\mathcal{O}_X \to \mathcal{O}_X$, $f \mapsto f^p$. It is well-known that the slope-stability of vector bundles may be destroyed under Frobenius pull-back. The points of $\mathcal{M}^s_X(r,d)$ that are destabilized by Frobenius pull-back can be classified using the  Harder-Narasimhan polygons (see Section \ref{sec2} for some details) of their pull-backs. The classification endows   $\mathcal{M}^s_X(r,d)$ (also $\mathcal{M}^s_X(r,\mathcal{L})$) with the so-called \emph{Frobenius stratification}. For certain moduli spaces determined by specific parameters such as $p,r,d,g$ and $\mathcal{L}$, many geometric properties of these strata have been studied. For further details on the  stratification of moduli spaces, we refer the reader to \cite{PazLange},  \cite{RussoBigas}, \cite{JRXY06}, \cite{Li19I}, \cite{Li20II}, \cite{LiZhang24}. 

In the case $\mathrm{char}(k)=2$, K. Joshi et al. \cite{JRXY06} completely describe the geometric properties of the Frobenius strata in $\mathcal{M}^{s}_X(2,d)$. For each $j\in \{1,2,...,g-1\}$ and Harder-Narasimhan polygon $\mathscr{P}_j$ (see the definition of $\mathscr{P}_j$ in Section \ref{sec2}), there are Frobenius strata $\mathcal{M}_j=  \big\{ [E]\in \mathcal{M}_X^s(2,d) \mid \mathrm{HNP}(F^*E) \succcurlyeq \mathscr{P}_j \big\}$. We denote the induced Frobenius strata 
$\mathcal{M}_j \cap \mathcal{M}^s_X(2,\mathcal{L})$ by $\mathcal{M}_j(\mathcal{L})$.

Rational curves in $\mathcal{M}_X^s(r,\mathcal{L})$  over $\mathbb{C}$ were studied by X. Sun in \cite{Sun05}.  Since the anti-canonical bundle $-K_{\mathcal{M}_X^s(r,\mathcal{L})}$ is ample, the degree of a rational curve $\phi: \mathbb{P}^1 \to \mathcal{M}_X^s(r,\mathcal{L})$ is well defined as $\mathrm{deg} (\phi^*(-K_{\mathcal{M}_X^s(r,\mathcal{L})}))$. He shows that any rational curve passing through the generic point of $\mathcal{M}_X^s(r,\mathcal{L})$ has degree at least $2r$. Hence, rational curves in $\mathcal{M}^s_X(r,\mathcal{L})$ with degree $2r$ are called \emph{minimal} rational curves. He also shows that a rational curve has degree $2r$ if and only if it is a \emph{Hecke curve} (see Section \ref{sec3}), unless $g = 3, r = 2$ and $d$ is even. The rational curves with degree less than $2r$, which are contained in a proper closed subvariety, are called \emph{small} rational curves. M. Liu \cite{Liu12} studies the codimension of the closed subvariety containing all small rational curves. J.-M. Drezet and M. S. Narasimhan \cite{DreNar89} show that $\mathrm{Pic}(\mathcal{M}^s_X(r,\mathcal{L})) = \mathbb{Z}\cdot \theta $ and $-K_{\mathcal{M}^s_X(r,\mathcal{L})}=2(r,d) \cdot \theta$, where $\theta$ is the generator ample line bundle. Thus, rational curve in $\mathcal{M}_X^s(r,\mathcal{L})$ of degree $2(r,d)$ are called \emph{lines}.  When $(r,d)=1$, X. Sun \cite{Sun05} completely determines all lines in $\mathcal{M}_X^s(r,\mathcal{L})$. Further studies of rational curves can be found in \cite{ChoeChungLee}, \cite{Hwang04}, \cite{Kollar13}, \cite{MokSun09}.

In the case $\mathrm{char}(k)=p>0$, N. Hoffmann \cite{Hoffmann12} shows that $\mathrm{Pic}(\mathcal{M}^s_X(r,\mathcal{L})) = \mathbb{Z}\cdot \mathcal{L}_{\theta}$, where $\mathcal{L}_{\theta}$ is the universal determinant line bundle. In this paper, we study \emph{Hecke curves} in Frobenius strata of $\mathcal{M}^s_X(2,\mathcal{L})$. Our main results are summarized in the following theorem.
\begin{theorem}
Let $k$ be an algebraically closed field of characteristic $2$, and let $X$ be a smooth projective curve of genus $g\geqslant 2$ over $k$. Let $\mathcal{M}^s_X(2,d)$ be the moduli space of stable vector bundles of rank $2$ and degree $d$ over $X$ and $\mathcal{M}_j:= \{ [E]\in \mathcal{M}_X^s(2,d) \mid \mathrm{HNP}(F^*E) \succcurlyeq \mathscr{P}_j \}$, $\mathcal{M}_j(\mathcal{L}):=\mathcal{M}^s_X(2,\mathcal{L}) \cap \mathcal{M}_j$ be the Frobenius strata,  $1\leqslant j \leqslant g-1$. Then
\begin{itemize}
    \item [(1)] There is no rational curve in $\mathcal{M}_{g-1}$;
    \item[(2)] For any rational curve $\phi: \mathbb{P}^1 \to \mathcal{M}^s_X(2,d)$, there exists $\mathcal{L} \in \mathrm{Pic}^d(X)$ such that $\phi(\mathbb{P}^1) \subseteq \mathcal{M}^s_X(2,\mathcal{L}) \subseteq \mathcal{M}^s_X(2,d)$;
    \item[(3)] For any $\mathcal{L}\in \mathrm{Pic}^d(X)$ and any $[E] \in \mathcal{M}_{g-2}(\mathcal{L}) \backslash \mathcal{M}_{g-1}(\mathcal{L})$, there exists a Hecke curve in $\mathcal{M}_{g-2}(\mathcal{L})\backslash \mathcal{M}_{g-1}(\mathcal{L})$ passing through $[E]$.
\end{itemize}
\end{theorem}

\section{Moduli space}\label{sec2}
Let $k$ be an algebraically closed field of characteristic $2$. Let $X$ be a smooth projective curve of genus $g\geqslant 2$ over $k$, and 
$\mathcal{M}^{s}_X(2,d)$ the moduli space of stable vector bundles of rank $2$ and degree $d$ over $X$.

\begin{definition}
Let $E$ be a vector bundle over $X$. The \emph{slope} of $E$ is the rational number $\mu(E)=\mathrm{deg}(E)/\mathrm{rank}(E)$.
$E$ is called \emph{stable} (resp. \emph{semistable}) if for any non-zero proper subbundle $F \subset E$,
$\mu(F) < \mu(E)$ (resp. $\mu(F) \leqslant  \mu(E)$).
\end{definition}
It is well-known that the Frobenius morphism $F:X \to X$, induced by $f \mapsto f^p$, may destroy the stability of vector bundles \cite{Gieseker73}. The instability can be measured in terms of \emph{Harder-Narasimhan} polygons (HNP): for a vector bundle $E$ on $X$, the Harder-Narasimhan filtration of $E$ is the unique filtration
\begin{equation}
\mathrm{HN}_{\bullet}(E): 0=E_0 \subset E_1 \subset E_2 \subset\cdots \subset  E_{n-1} \subset E_n = E, \nonumber
\end{equation}
such that $E_{i} / E_{i-1}(i=1,2,\dots, n)$ are semistable and
$\mu(E_1) >\mu(E_2/E_1) > \cdots > \mu(E_n /E_{n-1} )$. 
Connecting $(\mathrm{rank}(E_i),\mathrm{deg}(E_i))$ in the rank-degree coordinate plane successively and $(0,0)$ to $(\mathrm{rank}(E), \mathrm{degree}(E))$, the resulting convex polygon $\mathscr{P}(E)$ is called the $\emph{Harder-Narasimhan polygon}$ of $E$ (we refer to \cite{Shatz77} Section 3 for more details). In the case of rank $2$ and $\mathrm{char}(k)=2$, we denote the polygons defined in the following picture
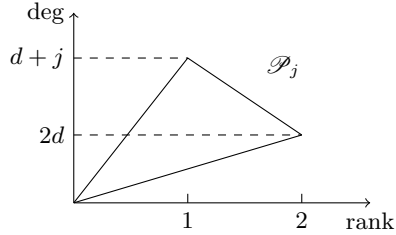
\newline
\begin{center}
\begin{tikzpicture}[scale=0.6]
\tikzstyle{every node}=[font=\small,scale=1]
\tikzstyle{arrow}=[->,>=Stealth]       	
\draw [->] (0,0)--(6.5,0) node [below] {rank} ;
\draw [->] (0,0)--(0,4.2) node [left] {deg}; 
\draw (2.5,0)--(2.5,0.2);
\draw (5,0)--(5,0.2);
\draw (2.5,0) node [below] {$1$};
\draw (5,0) node [below] {$2$};
\draw (0,1.5) node [left] {$2d$};
\draw [dashed] (0,1.5)--(5,1.5);
\draw [dashed] (0,3.2)--(2.5,3.2);
\draw (0,3.2) node [left] {$d+j$};
\draw (0,0)--(2.5,3.2)--(5,1.5);
\draw (0,0)--(5,1.5);
\draw (4,3)  node [right]{$\mathscr{P}_j$};
\end{tikzpicture}
\end{center}
by $\mathscr{P}_j$, where $j\in \mathbb{N}$. There exists a natural order  on $\{\mathscr{P}_j\}_{j\in \mathbb{N}}$: if $j_1\geqslant j_2$, we denote $\mathscr{P}_{j_1} \succcurlyeq \mathscr{P}_{j_2}$. K. Joshi et.\,al. \cite{JRXY06} completely classify the Frobenius destabilized locus in $\mathcal{M}^s_X(2,d)$. Explicitly, they define 
\begin{equation*}
\mathcal{M}_j:= \big\{ [E]\in \mathcal{M}_X^s(2,d) \mid \mathrm{HNP}(F^*E)\succcurlyeq \mathscr{P}_j \big\}
\end{equation*}
and prove that they are closed subvarieties of $\mathcal{M}^s_X(2,d)$
with $\dim_k \mathcal{M}_j= g+2(g-1-j)$ for $1 \leqslant j \leqslant g-1$.
Moreover, $\mathcal{M}_j$ satisfy the nested property: $\mathcal{M}_{j}=\overline{\mathcal{M}_{j}\backslash \mathcal{M}_{j+1}} = \bigsqcup_{i \geqslant j} \mathcal{M}_i \backslash \mathcal{M}_{i+1}$.
Consider the determinant morphism
\begin{eqnarray}
    \mathrm{det}: \mathcal{M}^s_X(2,d) &\longrightarrow &  \mathrm{Pic}^d(X) \nonumber \\
  \left[ E \right]  \quad \, &  \longmapsto &   \, \mathrm{det}(E). \nonumber
\end{eqnarray}
For any $\mathcal{L} \in \mathrm{Pic}^d(X)$, one has $\mathcal{M}^s_X(2,\mathcal{L}):=\mathrm{det}^{-1}(\mathcal{L})$.
Taking intersection induces the Frobenius stratification of $\mathcal{M}^s_X(2, \mathcal{L})$.We denote the induced Frobenius strata 
$\mathcal{M}_j \cap \mathcal{M}^s_X(2,\mathcal{L})$ by $\mathcal{M}_j(\mathcal{L})$, where $1 \leqslant j \leqslant g-1$.

Let $E\subseteq F$ be locally free sheaves of rank 2 on $X$. For an integer $\ell \geqslant 0$,
$E$ is said to be a submodule of $F$ with \emph{co-length} $\ell$ if $length(F/E)=\ell$. Equivalently, $F/E$ is a skyscraper sheaf of degree $\ell$ supported on finite points. The following lemmas are due to K. Joshi et al. \cite{JRXY06}:
\begin{lemma}[Subsection 4.2 in \cite{JRXY06}]\label{lemma3}
    Let $k$ be an algebraically closed field of characteristic $2$, $X$ a smooth projective curve of genus $g\geqslant 2$ over $k$, $L$ a line bundle on $X$ of degree $d$, $V$ a submodule of $F_*L$ with co-length $\ell$, where $1 \leqslant \ell \leqslant g-2$. Then $V$ is stable and $F^*V$ is not semistable.
\end{lemma}
\begin{lemma}[Subsection 4.3 in \cite{JRXY06}]\label{lemma4}
Let $k$ be an algebraically closed field of characteristic $2$, $X$ a smooth projective curve of genus $g\geqslant 2$ over $k$, $E$  a rank 2 degree $d$ stable vector bundle such that $F^*E$ is not semistable. Suppose the Harder-Narasimhan filtration of $F^*E$ is 
$$0= E_0\subset E_1 \subset E_2=F^*E$$
with $\mathrm{deg}(E_1)=d+j$, $1 \leqslant j \leqslant g-1$. Let $L:=F^*E/E_1$ be the quotient line bundle of degree $d-j$. Then, $E$ is a submodule of $F_*L$ with co-length $g-1-j$.
\end{lemma}
 
The particular case we will use in Section 4 is the following: for any $[E] \in \mathcal{M}_{g-2} \backslash \mathcal{M}_{g-1}$, there exists a line bundle $L\in \mathrm{Pic}^{d-g+2}(X)$ fitting into a short exact sequence
\begin{equation*}
    0 \to E \to F_*L \to \textbf{k}(x) \to 0,
\end{equation*}
where $\textbf{k}(x)$ denotes the skyscraper sheaf of degree $1$ supported at $\{x\}$.

\section{Hecke curves}\label{sec3}
Let $k$ be an algebraically closed field of arbitrary characteristic and $X$ be a smooth projective curve of genus $g \geqslant 2$  over $k$. Let $\mathcal{M}^s_X(r,\mathcal{L})$ be the moduli space of stable vector bundles with fixed determinant $\mathcal{L} \in \mathrm{Pic}^d(X)$. 
Since we will see in Section \ref{sec4} that every rational curve in $\mathcal{M}^s_X(r,d)$ is contained in $\mathcal{M}^s_X(r,\mathcal{L})$ for some $\mathcal{L}\in \mathrm{Pic}^d(X)$, we work directly with rational curves in $\mathcal{M}^s_X(r,\mathcal{L})$. A rational curve in $\mathcal{M}^s_X(r,\mathcal{L})$ is a non-trivial morphism $\phi:\mathbb{P}^1 \to \mathcal{M}^s_X(r,\mathcal{L})$. The degree of $\phi$ is defined as $\mathrm{deg}(\phi):=\mathrm{deg}(\phi^*(-K_{\mathcal{M}^s_X(r,\mathcal{L})}))$. In this section, we focus on the study of \emph{Hecke curves}.

\emph{Hecke transformations} provide a fundamental method for constructing rational curves passing through a fixed point $[E]\in\mathcal{M}^s_X(r,\mathcal{L})$. This section reviews this approach, drawing on results from Section 1 of \cite{Sun05}.
For a vector bundle $E$ over $X$, we denote its fibre at $x\in X$ by $E_x$, and let $\textbf{k}(x)$ denote the skyscraper sheaf of degree $1$ supported on $\{x\}$. For a subspace $W\subset E_x$, there are two types of modifications of $E$, known as \emph{Hecke transformation} (I) and \emph{Hecke transformation} (II) defined as follows:
\begin{itemize}
    \item [(I)] The \emph{Hecke transformation} (I) of $E$ along $W$ at $x \in X$ is defined as $$E^W:=\mathrm{Ker}(E\to (E_x/W)\otimes\textbf{k}(x))$$
    Explicitly, $E^W$ satisfies
    $$0\to E^W \xrightarrow{\phi} E \to (E_x/W)\otimes\textbf{k}(x)\to 0$$
    with $\phi_x(E^W_x)=W \subseteq E_x$.
    \newline
    \item[(II)] For $W^{\bot} \subset E_x^{\vee}$, the subspace annihilated by $W$, let $(E^{\vee})^{W^{\bot}}$ be the Hecke transformation (I) of  $E^\vee$ along $W^{\bot}$, \emph{i.e.} $(E^{\vee})^{W^{\bot}}$ satisfies 
    $$0\to(E^{\vee})^{W^{\bot}} \xrightarrow{\phi} E^{\vee} \to (E_x^{\vee}/W^{\bot})\otimes\textbf{k}(x)\to 0.$$
    Its dual bundle $\widetilde{E^W}:=((E^{\vee})^{W^{\bot}})^{\vee}$ satisfies 
    $$0 \to E \xrightarrow{\psi} \widetilde{E^W} \to (\widetilde{E_{\ x}^W}/(W^{\bot})^{\vee})\otimes \textbf{k}(x) \to 0$$
    with $\mathrm{Ker}(\psi_x: E_x \to \widetilde{E_{\ x}^W})=W$. We call $\widetilde{E^W}$ the  \emph{Hecke transformation} (II) of $E$ along $W$ at $x \in X$.
\end{itemize}
The above procedures actually construct vector bundles $E^W$ and $\widetilde{E^W}$ satisfying
$$E^W \subset E \subset \widetilde{E^W}$$ with degree $\mathrm{deg}(E^W)=\mathrm{deg}(E)-\dim_k(E_x/W)$ and $\mathrm{deg}(\widetilde{E^W})= \mathrm{deg}(E)+\dim_k(E_x/W)$.

For any $[E]\in\mathcal{M}^s_X(r,d)$ and $x \in X$, let $\zeta \subset E_x$ be a 1-dimensional subspace. Hecke transformation (I) produces a vector bundle $E^{\zeta}$ satisfying 
\begin{equation*}
    0 \to E^{\zeta} \to E \to (E_x/\zeta)\otimes \textbf{k}(x) \to 0.
\end{equation*}
Let $\iota: E_x^{\zeta}\to E_x$ be the morphism between fibres at $x$ induced by the sheaf injection $ E^{\zeta}\to E$. The kernel $\mathrm{ker}(\iota)$ is a $(r-1)$-dimensional subspace of $E_x^{\zeta}$. Let $\mathrm{Gr}(r-1,E_x^\zeta)\cong \mathbb{P}^{r-1}$ be the Grassmannian of $(r-1)$-dimensional subspaces of $E_x^\zeta$ and $\mathcal{H}$ be a line in $\mathrm{Gr}(r-1,E_x^\zeta)$ passing through the point $[\mathrm{ker}(\iota)]$. For any point $[l]\in \mathcal{H}$ corresponding to a $(r-1)$-dimensional subspace $l \subset E_x^{\zeta}$, Hecke transformation (II) yields a vector bundle $\widetilde{(E^{\zeta})^{l}}$ fitting into a short exact sequence
\begin{equation*}
    0 \to E^{\zeta} \to \widetilde{(E^{\zeta})^{l}} \to (\widetilde{(E^{\zeta})^{l}_x}/(l^{\bot})^{\vee})\otimes \textbf{k}(x) \to 0
\end{equation*}
and for $l = \mathrm{ker}(\iota)$, one has $(\widetilde{E^{\zeta})^{\mathrm{ker}(\iota)}} \cong E$. Note that $\{\widetilde{(E^{\zeta})^{l}} \mid [l] \in \mathcal{H} \}$ are still vector bundles of rank $r$ and degree $d$. Suppose that $\{\widetilde{(E^{\zeta})^{l}}\}$ are stable, then the above construction yields a rational curve $\{ \widetilde{(E^{\zeta})^{l}}\mid [l]\in \mathcal{H} \}$ passing through $[E]$ in $\mathcal{M}^s_X(r,\mathcal{L})$, which we call a \emph{Hecke curve}. A study of Hecke curves can be found in \cite{Hwang02}.

In the case $k=\mathbb{C}$ and $X$ has genus at least 3, $r\geqslant 2$, for any $\mathcal{L}\in \mathrm{Pic}^d(X)$, X. Sun \cite{Sun05} showed any rational curve passing through the generic point in $\mathcal{M}^s_X(r,\mathcal{L})$ has degree at least $2r$. So the rational curves of degree $2r$ are also called \emph{minimal} rational curves. The following theorem is due to him:

\begin{theorem}[Theorem 1 \cite{Sun05}]
If $k=\mathbb{C}$ and $g \geq 3$, then any rational curve $\phi: \mathbb{P}^1 \to \mathcal{M}^s_X(r,\mathcal{L})$ passing through the generic point has degree at least $2r$. It has degree $2r$ if and only if it is a Hecke curve unless $g=3, r=2$ and $d$ is even.
\end{theorem}

\section{Main results}\label{sec4}
The aim of this section is to prove the following theorem: 
\begin{theorem}
Let $k$ be an algebraically closed field of characteristic $2$, and let $X$ be a smooth projective curve of genus $g\geqslant 2$ over $k$. Let $\mathcal{M}^s_X(2,d)$ be the moduli space of stable vector bundles of rank $2$ and degree $d$ over $X$ and $\mathcal{M}_j:=\{ [E]\in \mathcal{M}_X^s(2,d) \mid \mathrm{HNP}(F^*E) \succcurlyeq \mathscr{P}_j \}$, $\mathcal{M}_j(\mathcal{L}):=\mathcal{M}^s_X(2,\mathcal{L}) \cap \mathcal{M}_j$ be the Frobenius strata,  $1\leqslant j \leqslant g-1$. Then
\begin{itemize}
    \item [(1)] There is no rational curve in $\mathcal{M}_{g-1}$;
    \item[(2)] For any rational curve $\phi: \mathbb{P}^1 \to \mathcal{M}^s_X(2,d)$, there exists $\mathcal{L} \in \mathrm{Pic}^d(X)$ such that $\phi(\mathbb{P}^1) \subseteq \mathcal{M}^s_X(2,\mathcal{L}) \subseteq \mathcal{M}^s_X(2,d)$;
    \item[(3)] For any $\mathcal{L}\in \mathrm{Pic}^d(X)$ and any $[E] \in \mathcal{M}_{g-2}(\mathcal{L}) \backslash \mathcal{M}_{g-1}(\mathcal{L})$, there exists a Hecke curve in $\mathcal{M}_{g-2}(\mathcal{L})\backslash \mathcal{M}_{g-1}(\mathcal{L})$ passing through $[E]$.
\end{itemize}
\end{theorem}
\begin{proof}
(1) In the case of characteristic $2$, according to \cite[Theorem 2.5]{Li14}, the isomorphism
\begin{eqnarray}
S^s_{Frob}: \mathrm{Pic}^{d-g+1}(X) & \to & \mathcal{M}_{g-1} \nonumber \\
    \left[L\right] \qquad & \to & \left[F_*L\right] \nonumber
\end{eqnarray}
implies that $\mathcal{M}_{g-1}$ is isomorphic to an abelian variety, which cannot contain any rational curves.

(2) Suppose $\phi: \mathbb{P}^1 \to \mathcal{M}^s_X(2,d)$ is a rational curve. Consider the determinant morphism 
\begin{eqnarray}
    \mathrm{det}: \mathcal{M}^s_X(2,d) &\longrightarrow &  \mathrm{Pic}^d(X) \nonumber \\
  \left[ E \right]  \quad \, &  \longmapsto &   \, \mathrm{det}(E). \nonumber
\end{eqnarray}
The image of $\mathrm{det} \circ \phi$ is constant since $\mathrm{Pic}^d(X)$ is an abelian variety. Thus, any rational curve is contained in $\mathcal{M}^s_X(2,\mathcal{L})$ for some $\mathcal{L} \in \mathrm{Pic}^d(X)$.

(3) For any $[{E}] \in \mathcal{M}_{g-2}(\mathcal{L})\backslash\mathcal{M}_{g-1}(\mathcal{L})$, according to Lemma \ref{lemma4}, there exists $L \in \mathrm{Pic}^{d-g+2}(X)$ such that 
\begin{equation}\label{4.1}
    0 \to E  \xrightarrow{f} F_*L  \to  \textbf{k}(x) \to 0, 
\end{equation}
is exact. Then, $f$ induces an exact sequence of vector spaces  
\begin{equation*}
    E_x  \xrightarrow{f_x} (F_*L)_x \to \textbf{k}(x)_x \to 0,
\end{equation*}
where $\dim_k\textbf{k}(x)_x=1$. 
For $\zeta:=  \ \mathrm{Ker}(f_x) \subset E_x$, Hecke transformation (I) gives a vector bundle $E^{\zeta}$ satisfying 
\begin{equation*}
    0 \to E^{\zeta} \xrightarrow{q} E \to (E_x/\zeta)\otimes \textbf{k}(x) \to 0. 
\end{equation*}
According to the construction, $q_x(E^{\zeta}_x)=\zeta = \mathrm{Ker}(f_x)$ and $g_x:=f_x \circ q_x: E^{\zeta}_x \to (F_*L)_x$
is zero. We denote the unique maximal ideal of $\mathcal{O}_{X,x}$ by $\mathfrak{m}_x$. To avoid confusion, we denote the localization at a point $x\in X$ by $(-)_{\mathfrak{p}_x}$. Locally, the image of 
$g_{\mathfrak{p}_x}:E_{\mathfrak{p}_x}^{\zeta} \to (F_*L)_{\mathfrak{p}_x}$
is contained in $\mathfrak{m}_x \cdot (F_*L)_{\mathfrak{p}_x}$.
Hence, $g:=f\circ q:E^\zeta \to F_*L$ factors through $F_*L(-x) \to F_*L$. The following commutative diagram 
$$\xymatrix{
0 \ar[r] & E^\zeta \ar[r]^{g} \ar[d]^{\cong} & F_*L \ar[r]\ar@{=}[d] & (F_*L)_x \otimes \textbf{k}(x) \ar[r]\ar[r]\ar@{=}[d] & 0\\
0\ar[r]  &F_*L(-x) \ar[r] & F_*L \ar[r] & (F_*L)_x\otimes \textbf{k}(x) \ar[r] & 0}$$
implies $E^\zeta \cong F_*L(-x)$.

For any point $[l] \in \mathrm{Gr}(1,E_x^\zeta)\cong \mathbb{P}^1$, Hecke transformation (II) of $E^\zeta$ along $l\subset E_x^\zeta$ yields a vector bundle $\widetilde{(E^{\zeta})^l}$ satisfying
\begin{equation}\label{applyfunctor}
    0 \to E^\zeta \xrightarrow{r} \widetilde{(E^{\zeta})^l}\xrightarrow{s} (\widetilde{(E^{\zeta})_x^l}/(l^{\bot})^{\vee})\otimes \textbf{k}(x) \to 0.
\end{equation}
Note that for $l = \mathrm{ker}(q_x)$, we have $\widetilde{(E^\zeta)^{l}} \cong E$ by Section \ref{sec3}. 

For an injection of $k$-vector space $i: \widetilde{(E^{\zeta})_x^l}/(l^{\bot})^{\vee} \to (F_*L)_x$, if there exists an injection $h_{\mathfrak{p}_x}$ making the following diagram of $\mathcal{O}_{X,x}$-modules commute, 
$$\xymatrix{
0 \ar[r] & E_{\mathfrak{p}_x}^{\zeta} \cong F_*L(-x)_{\mathfrak{p}_x} \ar[r]^{\qquad r_{\mathfrak{p}_x}} \ar[d]^{\cong} & \widetilde{(E^\zeta)^l}_{\mathfrak{p}_x} \ar[r]^{s_{\mathfrak{p}_x}\quad} \ar@{-->}[d]^{h_{\mathfrak{p}_x}}& \widetilde{(E^{\zeta})_x^l}/(l^{\bot})^{\vee} \ar[r]\ar[d]^{i} & 0\\
0\ar[r]  &F_*L(-x)_{\mathfrak{p}_x} \ar[r]  & (F_*L)_{\mathfrak{p}_x} \ar[r]^{t_{\mathfrak{p}_x}} & (F_*L)_x \ar[r] & 0
}$$
we have an injection 
$$h:\widetilde{(E^\zeta)^l} \to F_*L$$ 
of sheaves of $\mathcal{O}_{X}$-modules and $\mathrm{deg}(F_*L) = \mathrm{deg}(\widetilde{(E^\zeta)^l})+1$.

Since the problem is local, we may assume $(F_*L)_{\mathfrak{p}_x} \cong F_*(k[[t]])$ is a rank 2 free $k[[t^2]]$-module generated by $\{1,t\}$ and $F_*L(-x)_{\mathfrak{p}_x}=\mathfrak{m}_x \cdot (F_*L)_{\mathfrak{p}_x}$ generated by $t^2 \cdot \{1,t\} = \{t^2, t^3\}$. Then $(F_*L)_x$ is a $k$-vector space spanned by $\{\overline{1}, \overline{t} \}$. Denote $\boldsymbol{u}:=r_{\mathfrak{p}_x}(t^2)$ and $\boldsymbol{v}:=r_{\mathfrak{p}_x}(t^3)$.  Suppose $\widetilde{(E^\zeta)^l}_x/(l^{\bot})^{\vee}$ is spanned by  $\overline{\boldsymbol{w}}$ and $\boldsymbol{w}$ is a preimage of $\overline{\boldsymbol{w}}$ in $\widetilde{(E^\zeta)^l}_{\mathfrak{p}_x}$. Then, $\widetilde{(E^\zeta)^l}_{\mathfrak{p}_x}$ can be expressed as 
$$\widetilde{(E^\zeta)^l}_{\mathfrak{p}_x} = k[[t^2]] \cdot \boldsymbol{u}   \oplus k[[t^2]] \cdot \boldsymbol{v}  + k\cdot \boldsymbol{w}.$$
Let $i: \overline{\boldsymbol{w}} \mapsto \overline{a+bt} \in (F_*L)_x$, where $a,b \in k$. Then $a+bt$ is a representative element of $t_{\mathfrak{p}_x}^{-1}(\overline{a+bt})$ in $(F_*L)_{\mathfrak{p}_x}$. We can construct the natural embedding 
\begin{equation*}
    h_{\mathfrak{p}_x}: \boldsymbol{u} \mapsto t^2, \quad  \boldsymbol{v} \mapsto t^3, \quad  \boldsymbol{w} \mapsto a+bt.
\end{equation*}
It is obvious that $h_{\mathfrak{p}_x}$ makes the diagram commute. Since $$\mathrm{deg}(F_*L/\widetilde{(E^\zeta)^l})=1,$$
by Lemma \ref{lemma3}, $\widetilde{(E^\zeta)^l}$ is stable and $F^*(\widetilde{(E^\zeta)^l})$ is not semistable. 

Finally, we prove the invariance of the stratum, namely, that
$[\widetilde{(E^\zeta)^l}] \in \mathcal{M}_{g-2}(\mathcal{L})\backslash\mathcal{M}_{g-1}(\mathcal{L})$
for every $[l]\in \mathrm{Gr}(1,E_x^\zeta)$.
Set $W := \widetilde{(E^\zeta)^l}$, and let
$$
0 \subsetneq W_1 \subsetneq F^*W
$$
be the Harder-Narasimhan filtration of $F^*W$, where
$\deg(W_1) = d + j'$ and the quotient line bundle
$Q := F^*W / W_1$ has degree $d - j'$
($1 \leqslant j' \leqslant g-1$).
By the injection $W \hookrightarrow F_*L$ and adjunction,
there exists a nonzero morphism $F^*W \to L$.
When $g \geqslant 3$, we have
$$
\deg(W_1) - \deg(L)
= (d+j') - (d-g+2)
= j' + g - 2 \geqslant 2 > 0.
$$
Hence, the restriction of this morphism to $W_1$
must vanish, and it therefore induces a nonzero
morphism $Q \to L$.
Since a nonzero morphism between line bundles
on a smooth projective curve implies the corresponding
inequality of degrees, we obtain
$$
\deg(L) \geqslant \deg(Q)
\implies d - g + 2 \geqslant d - j'
\implies j' \geqslant g - 2.
$$

Suppose that $j' = g-1$.
By Lemma \ref{lemma4}, the remaining colength
is $g-1-j'=0$, which implies that
$W \cong F_*Q$, where $\deg(Q) = d-g+1$.
In this case, the injection $E^\zeta \hookrightarrow W$
induces an injection
$F_*L(-x) \hookrightarrow F_*Q$.
By adjunction, we have
$$
\mathrm{Hom}(F_*L(-x), F_*Q)
\cong
\mathrm{Hom}(F^*(F_*L(-x)), Q).
$$
Moreover, the line subbundle in the canonical filtration
of $F^*(F_*L(-x))$ has degree
$$
\deg(L(-x) \otimes \Omega_X^1)
= d+g-1 > \deg(Q).
$$
Consequently, every morphism from this line subbundle
to $Q$ must vanish, yielding an isomorphism
$$
\mathrm{Hom}(F_*L(-x), F_*Q)
\cong
\mathrm{Hom}(L(-x), Q).
$$
It follows that the injection $E^\zeta \hookrightarrow W$
must be induced by a nonzero morphism
$L(-x) \to Q$. However,
$$
\deg(L(-x))
= \deg(L) - 1
= d-g+1
= \deg(Q).
$$
Since a nonzero morphism between line bundles
of the same degree is necessarily an isomorphism,
we deduce that
$$
E^\zeta \cong F_*L(-x) \cong F_*Q \cong W.
$$
This contradicts
$\deg(W) - \deg(E^\zeta) = 1$.
Therefore, we have $j' = g-2$ and obtain a Hecke curve $\{\widetilde{(E^\zeta)^l} \mid  [l] \in \mathrm{Gr}(1,E_x^\zeta)\}$ in $\mathcal{M}_{g-2}(\mathcal{L})\backslash\mathcal{M}_{g-1}(\mathcal{L})$ passing through $[E]$.

\end{proof}

\end{document}